\documentclass[10pt]{elsarticle}
\usepackage[cp1251]{inputenc}
\usepackage[english]{babel}
\usepackage{amsmath}
\usepackage{amssymb}
\usepackage{amsfonts}
\usepackage{amsthm}
\usepackage{mathtools}
\usepackage{dsfont}
\usepackage{rotating}
\usepackage{graphicx}
\usepackage{floatflt,epsfig}
\usepackage{lineno,hyperref}
\usepackage{enumerate}
\usepackage{colortbl}
\usepackage{array,tabularx,tabulary,booktabs}
\usepackage{longtable}
\usepackage{multirow}
\usepackage{wrapfig}
\usepackage{subcaption}
\usepackage{pdflscape}
\usepackage[table]{xcolor}
\usepackage[left=1in,right=1in,top=1in,bottom=1in]{geometry}

\newcolumntype{^}{>{\currentrowstyle}}

\journal{Designs, Codes and Cryptography}
\newtheorem{theorem}{Theorem}
\newtheorem{corollary}{Corollary}
\newtheorem{proposition}{Proposition}

\newtheorem{problem}{Problem}

\begin{document}
\renewcommand{\abstractname}{Abstract}
\renewcommand{\refname}{References}
\renewcommand{\tablename}{Table}
\renewcommand{\arraystretch}{0.9}
\thispagestyle{empty}
\sloppy

\begin{frontmatter}
\title{Non-canonical maximum cliques without a design structure in the block graphs of 2-designs}

\author[01]{Sergey Goryainov}
\ead{sergey.goryainov3@gmail.com}

\author[02,03,04]{Elena V. Konstantinova}
\ead{e\_konsta@ctgu.edu.cn}

\address[01] {School of Mathematical Sciences, Hebei International Joint Research Center for Mathematics and Interdisciplinary Science, Hebei Key Laboratory of Computational Mathematics and Applications\\Hebei Normal University, Shijiazhuang  050024, Hebei province, P.R. China}
\address[02] {Three Gorges Mathematical Research Center, China Three Gorges University, \\ 8 University Avenue, Yichang 443002, Hubei Province, P.R. China}
\address[03]{Sobolev Institute of Mathematics, Ak. Koptyug av. 4, Novosibirsk, 630090, Russia}
\address[04]{Novosibirsk State University, Pirogova str. 2, Novosibirsk, 630090, Russia}


\begin{abstract}
In this note we answer positively a question of Chris Godsil and Karen Meagher on the existence of a 2-design whose block graph has a non-canonical maximum clique without a design structure.
\end{abstract}

\begin{keyword}
block design; canonical cliques; non-canonical cliques; subdesign 
\vspace{\baselineskip}
\MSC[2020] 05E30 \sep 05B05 
\end{keyword}
\end{frontmatter}

\section{Introduction}\label{sec:Intro}
In \cite[Chapter 5]{GM15}, Chris Godsil and Karen Meagher proposed the study of Erd\H{o}s-Ko-Rado (EKR) properties of strongly regular graphs. In particular, they proposed the study of EKR properties of the block graphs of 2-$(n,m,1)$ designs. The notion of `intersecting' is natural: two blocks (two vertices of the block graph) are called \emph{intersecting} if they intersect as sets. The upper bound $\frac{n-1}{m-1}$ for the size of a maximum intersecting set of blocks (equivalently, a maximum clique in the block graph) is given by the Delsarte bound. Moreover, this bound is tight for the block graph of a 2-$(n,m,1)$ design as the set of all blocks containing a fixed point gives a clique of size $\frac{n-1}{m-1}$; a clique of this size and form is called \emph{canonical}. It follows from \cite[Corollary 5.3.5]{GM15} that the block graph of a 2-$(n,m,1)$ design with $n > m^3-2m^2+2m$ has only canonical maximum cliques. If $n \le m^3-2m^2+2m$, this characterisation may fail (non-canonical maximum cliques may exist). Further, it follows from \cite[Exercise 5.7]{GM15} that the vertices of each non-canonical clique (when such a clique exists) in the block graph of a 2-$(n,m,1)$ design with $n = m^3-2m^2+2m$ necessarily form a 2-$(m^2-m+1,m,1)$ subdesign (which is a projective plane of order $m-1$). However, the authors of \cite{GM15} said that it is not clear if this a result of a wider phenomenon (see \cite[p. 310]{GM15}). They then posed the following problem.

\begin{problem}[{\cite[Problem 16.3.2]{GM15}}]
When the block graph of a design has maximum cliques that are not canonical, are the non-canonical cliques isomorphic to smaller designs?    
\end{problem}

In this note we answer this question and show that it is not necessarily true that the non-canonical cliques are isomorphic to smaller designs.

\begin{theorem}\label{thm:main}
There exists a $2$-$(66,6,1)$ design (having exactly $143$ blocks) such that the following statements hold.\\
{\rm (1)} The block graph of the design has exactly $80$ maximum cliques: $66$ canonical and $14$ non-canonical ones.\\
{\rm (2)} Each of the $14$ non-canonical cliques does not have a design structure.
\end{theorem}

To the best of our knowledge, the corresponding block graph with 143 vertices is the smallest (in the sense of number of vertices) known with the property of having non-canonical cliques without a design structure. We have checked with use of Magma~\cite{BCP97} that the smaller block graphs of the designs from~\cite{BBT03},~\cite{C99},~\cite[p. 34, Table 1.34]{CD07},~\cite{CDRS19},~\cite{D14}, \cite{K02a},~\cite{K02b},~\cite{K04},~\cite{K06},~\cite{MR85},  and~\cite{MN21} have only canonical maximum cliques. Note that the data on most of these designs were taken from~\cite{B} and~\cite{K}. Let us mention that the design from Theorem \ref{thm:main} is not unique. The database given in \cite{K} contains at least three $2$-$(66,6,1)$ designs. Each of the block graphs of these designs has exactly 14 non-canonical cliques (all without a design structure); one of these designs is isomorphic to the design from Theorem \ref{thm:main}. 

This note is organised as follows. In Section \ref{sec:prelim}, we give necessary definitions and preliminary results. In Section \ref{sec:main}, we give a proof of Theorem \ref{thm:main} by analysing a design with a known symmetric description. In Section \ref{sec:discussion}, we discuss consequences of Theorem \ref{thm:main} and further research directions. In Appendix, we give a description and discuss two more designs with the same parameters, whose block graphs have non-canonical cliques without a design structure.

\section{Preliminaries}\label{sec:prelim}
In this section, we give necessary definitions and preliminary results.

\subsection{Block designs and their graphs}
For the background and more details, we refer to \cite[Section 5.3]{GM15}.

A 2-$(n,m,1)$ \emph{design} is a collection of $m$-sets of an $n$-set with the property that every pair from the $n$-set is in exactly one $m$-set. The elements of the $n$-set are called the \emph{points} of the design. The $m$-sets are called the \emph{blocks} of the design. It is well known that the number of blocks in a 2-$(n,m,1)$ design is $\frac{n(n-1)}{m(m-1)}$ and each point occurs in exactly $\frac{n-1}{m-1}$ blocks.  

The \emph{block graph of a} 2-$(n,m,1)$ \emph{design} is the graph with the blocks of the design as the vertices in which two blocks are adjacent if and only if they intersect.

A $k$-regular graph with $v$ vertices is called \emph{strongly regular} with parameters $(v,k,\lambda,\mu)$ if any two adjacent vertices in this graph have exactly $\lambda$ common neighbours and any two distinct nonadjacent vertices in this graph have exactly $\mu$ common neighbours. The block graph of a 2-$(n,m,1)$ design is known to be strongly regular with parameters 
$$
\left(
\frac{n(n-1)}{m(m-1)},
\frac{m(n-m)}{m-1},
(m-1)^2 + \frac{n-1}{m-1}-2,
m^2
\right)
$$
and to have smallest eigenvalue $-m$. 

Let $\theta$ be the smallest eigenvalue of a $k$-regular strongly regular graph $\Gamma$. Philippe Delsarte proved \cite{D73} that the clique number of $\Gamma$ is at most $1-\frac{k}{\theta}$. Applied to the block graph of a 2-$(n,m,1)$ design that is not symmetric, the Delsarte bound gives the value $\frac{n-1}{m-1}$. This bound is tight for the block graph of a 2-$(n,m,1)$ design (that is not symmetric) as the set of all blocks containing a fixed point gives a clique of size $\frac{n-1}{m-1}$; a clique of this size and form is called \emph{canonical}. 

\subsection{A $2$-$(66,6,1)$ design}\label{sec:design}
In this section, we discuss a 2-$(66,6,1)$ design constructed in \cite{D80}. To the best of our knowledge, this construction does not extend to an infinite family. The construction of this design can also be found in \cite[p. 76]{CD07}. In \cite{A23}, we got a more detailed description of the construction. Let us provide this detailed description here. The point set is $\mathcal{P} = \mathbb{Z}_{13} \times (\mathbb{Z}_3 \cup \{a,b\}) \cup \{\infty\}$. To define the block set, consider the following eleven basic blocks:
\begin{align*}
B_1&=\{2_0,5_0,4_1,9_1,0_a,6_a\},\\
B_2&=\{1_0,2_0,6_0,12_2,5_b,8_b\},\\
B_3&=\{6_1,2_1,12_2,1_2,0_a,5_a\},\\                
B_4&=\{3_1,6_1,5_1,10_0,2_b,11_b\},\\
B_5&=\{5_2,6_2,10_0,3_0,0_a,2_a\},\\               
B_6&=\{9_2,5_2,2_2,4_1,6_b,7_b\},\\
B_7&=\{7_0,9_0,10_1,1_2,3_a,4_b\},\\              
B_8&=\{2_a,6_a,5_a,4_b,12_b,10_b\},\\
B_9&=\{8_1,1_1,4_2,3_0,9_a,12_b\},\\            
B_{10}&=\{11_2,3_2,12_0,9_1,1_a,10_b\},\\
B_{11}&=\{\infty,0_0,0_1,0_2,0_a,0_b\}.   
\end{align*}
The 143 blocks of the block set $\mathcal{B}$ of the design are then obtained by developing modulo 13 the $\mathbb{Z}_{13}$-components of all points in the basic blocks. More precisely, for any $e \in \mathbb{Z}_{13}$ and basic block $B_i$, denote by $B_i^e$ the set of points obtained from the points of $B_i$ by adding $e$ modulo 13 to the $\mathbb{Z}_{13}$-component of each non-infinity point of $B_i$. Then, for the block set $\mathcal{B}$ of the design, we have:
$$
\mathcal{B} = \bigcup_{i = 1}^{11}\bigcup\limits_{e \in \mathbb{Z}_{13}} B_i^e.
$$
We find by Magma  \cite{BCP97} that the automorphism group of this design is a nonabelian group of order 39 and coincides with the automorphism group of the corresponding block graph. It can be generated by
\begin{align*}
\pi_1 = &
(0_0~10_1~1_2)
(1_0~0_1~10_2)
(2_0~3_1~6_2)
(3_0~6_1~2_2)
(4_0~9_1~11_2)
(5_0~12_1~7_2)\\
&
(6_0~2_1~3_2)
(7_0~5_1~12_2)
(8_0~8_1~8_2)
(9_0~11_1~4_2)
(10_0~1_1~0_2)
(11_0~4_1~9_2)
(12_0~7_1~5_2)\\
& 
(0_a~10_a~1_a)
(2_a~3_a~6_a)
(4_a~9_a~11_a)
(5_a~12_a~7_a)
(0_b~10_b~1_b)
(2_b~3_b~6_b)
(4_b~9_b~11_b)
(5_b~12_b~7_b)
\end{align*}
and
\begin{align*}
\pi_2 = &
(0_0~1_0~2_0~3_0~4_0~5_0~6_0~7_0~8_0~9_0~10_0~11_0~12_0)\\
&
(0_1~1_1~2_1~3_1~4_1~5_1~6_1~7_1~8_1~9_1~10_1~11_1~12_1)\\
&
(0_2~1_2~2_2~3_2~4_2~5_2~6_2~7_2~8_2~9_2~10_2~11_2~12_2)\\
&
(0_a~1_a~2_a~3_a~4_a~5_a~6_a~7_a~8_a~9_a~10_a~11_a~12_a)\\
&
(0_b~1_b~2_b~3_b~4_b~5_b~6_b~7_b~8_b~9_b~10_b~11_b~12_b).
\end{align*}
The orbits of this group on the set of points are 
$$\{i_j~:~ i \in \mathbb{Z}_{13}, j \in \mathbb{Z}_3\},$$ 
$$\{i_a~:~ i \in \mathbb{Z}_{13}\},$$ 
$$\{i_b~:~ i \in \mathbb{Z}_{13}\},$$ 
$$\{ \infty \}$$
and the orbits on the set of blocks are
$$\bigcup\limits_{e \in \mathbb{Z}_{13}} B_1^e \cup 
\bigcup\limits_{e \in \mathbb{Z}_{13}} B_3^e \cup
\bigcup\limits_{e \in \mathbb{Z}_{13}} B_5^e,$$
$$
\bigcup\limits_{e \in \mathbb{Z}_{13}} B_2^e \cup 
\bigcup\limits_{e \in \mathbb{Z}_{13}} B_4^e \cup 
\bigcup\limits_{e \in \mathbb{Z}_{13}} B_6^e, 
$$
$$
\bigcup\limits_{e \in \mathbb{Z}_{13}} B_7^e \cup 
\bigcup\limits_{e \in \mathbb{Z}_{13}} B_9^e \cup 
\bigcup\limits_{e \in \mathbb{Z}_{13}} B_{10}^e, 
$$
$$
\bigcup\limits_{e \in \mathbb{Z}_{13}} B_8^e, 
$$
$$
\bigcup\limits_{e \in \mathbb{Z}_{13}} B_{11}^e. 
$$

It is clear that the symbols $a$ and $b$ can be naturally used to characterise the orbits on the points and the blocks. This is the reason why we consider them and the symbols $\{0,1,2\}$ separately.

\section{Proof of Theorem \ref{thm:main}}\label{sec:main}
To prove Theorem \ref{thm:main}, we give explicit examples of non-canonical maximum cliques in the block graph of the 2-(66,6,1) design $(\mathcal{P},\mathcal{B})$ described in Section \ref{sec:design} and show that each of these cliques does not form a design.

Consider the set of blocks 
$$
C_1 = \{B_1^{11},B_2^1,B_3^1,B_4^{10},B_5^{10},B_6^{11},B_7^6,B_9^{12},B_{10}^4,B_{11}^0,B_{11}^2,B_{11}^3,B_{11}^7\},
$$
where
\begin{align*}
B_1^{11} &= \{0_0,3_0,2_1,7_1,11_a,4_a\},\\
B_2^1 &= \{2_0,3_0,7_0,0_2,6_b,9_b\},\\
B_3^1 &= \{7_1,3_1,0_2,2_2,1_a,6_a\},\\
B_4^{10} &= \{0_1,3_1,2_1,7_0,12_b,8_b\},\\
B_5^{10} &= \{2_2,3_2,7_0,0_0,10_a,12_a\},\\
B_6^{11} &= \{7_2,3_2,0_2,2_1,4_b,5_b\},\\
B_7^6 &= \{0_0,2_0,3_1,7_2,9_a,10_b\},\\
B_9^{12} &= \{7_1,0_1,3_2,2_0,8_a,11_b\},\\
B_{10}^4 &= \{2_2,7_2,3_0,0_1,5_a,1_b\},\\
B_{11}^0 &= \{\infty,0_0,0_1,0_2,0_a,0_b\},\\
B_{11}^2 &= \{\infty,2_0,2_1,2_2,2_a,2_b\},\\
B_{11}^3 &= \{\infty,3_0,3_1,3_2,3_a,3_b\},\\
B_{11}^7 &= \{\infty,7_0,7_1,7_2,7_a,7_b\}.
\end{align*}

\begin{proposition}\label{prop:C1}
The set $C_1$ is a non-canonical maximum clique in the block graph of the design $(\mathcal{P},\mathcal{B})$.     
\end{proposition}
\begin{proof}
One can check that any two blocks in $C_1$ intersect. Indeed, these 13 blocks (restricted to the points $\{\infty\} \cup \{n_x \mid n \in \{0, 2, 3, 7\}, x \in \{0, 1, 2\}$)
form a $2$-$(13, 4, 1)$ design, a projective plane of order 3. So, we have a 2-design,
and for each block a unique 2-point extension.    
\end{proof}

For any $e \in \mathbb{Z}_{13}$, put 
$$C_1^e = \{B^e \mid B \in C_1\}.$$

\begin{corollary}\label{cor:13cliques}
For any $e \in \mathbb{Z}_{13}$, the set $C_1^e$ forms a non-canonical maximum clique in the block graph of the design $(\mathcal{P},\mathcal{B})$.  
\end{corollary}
\begin{proof}
If follows from Proposition \ref{prop:C1} and the structure of the block set of the design $(\mathcal{P},\mathcal{B})$.    
\end{proof}
\noindent Thus, Corollary \ref{cor:13cliques} gives 13 non-canonical maximum cliques.

Consider the set of blocks
$$C_2=\bigcup\limits_{e \in \mathbb{Z}_{13}} \{B_8^e\},$$
obtained by developing modulo 13 the $\mathbb{Z}_{13}$-components of all points in the basic block $B_8$. 

\begin{proposition}\label{prop:C2}
The set $C_2$ is a non-canonical maximum clique in the block graph of the design $(\mathcal{P},\mathcal{B})$.     
\end{proposition}
\begin{proof}
Let us show that, for any distinct $e_1,e_2 \in \mathbb{Z}_{13}$, the blocks $B_8^{e_1}$ and $B_8^{e_2}$ intersect. Recall that $B_8 = \{2_a,6_a,5_a,4_b,12_b,10_b\}$. Let $R = \{2,6,5\}$, then $2R = \{4,12,10\}$. It is clear that the block $B_8$ can be obtained by adding the subscript $a$ to the elements from $R$, adding the subscript $b$ to the elements from $2R$ and joining these two 3-sets. Note that $R-R = 3\cdot\{0\} + 1\cdot S_{13}$ and $2R - 2R = 3\cdot\{0\} + 1\cdot N_{13}$, where $S_{13}$ and $N_{13}$ are the sets of squares and non-squares in $\mathbb{Z}_{13}^*$, respectively. In other words, every square (resp. non-square) from $\mathbb{Z}_{13}^*$ can be uniquely expressed as the difference of two distinct elements from $R$ (resp. $2R$). Let $i+R$ and $j+R$ be two translations of $R$ for some distinct $i,j \in \mathbb{Z}_{13}$. Note that, for any $d \in \mathbb{Z}_{13}$, we have $d \in (i+R) \cap (j+R)$ if and only if there exist $r_1,r_2 \in R$ such that $d = i + r_1 = j + r_2$, that is, if and only if there exist $r_1,r_2 \in R$ such that $i-j = r_2 - r_1$. This implies that $|(i+R) \cap (j+R)| = 1$ if $i-j$ is a square in $\mathbb{Z}_{13}^*$ and $|(i+R) \cap (j+R)| = 0$, otherwise. Similarly, we have $|(i+2R) \cap (j+2R)| = 1$ if $i-j$ is a non-square in $\mathbb{Z}_{13}^*$ and $|(i+2R) \cap (j+2R)| = 0$, otherwise. Thus, for any distinct $e_1,e_2 \in \mathbb{Z}_{13}$, the blocks $B_8^{e_1}$ and $B_8^{e_2}$ intersect.  
\end{proof}
Note that every point from the 26-element union of the blocks of $C_2$ belongs to exactly three blocks from $C_2$.

We then verify by Magma \cite{BCP97} that the block graph of $(\mathcal{P},\mathcal{B})$ has only these 14 non-canonical maximum cliques. The 66 canonical cliques split into four orbits of length 39, 13, 13 and 1, and the 14 non-canonical cliques split into two orbits of length 13 and 1.

Finally, a 2-(39,6,1) design (the total number of points in the union of the blocks from $C_1$ is 39) and a 2-(26,6,1) design (the total number of points in the union of the blocks from $C_2$ is 26) do not exist since $\frac{n-1}{m-1} \notin \mathbb{N}$ (for the former) and $\frac{n(n-1)}{m(m-1)} \notin \mathbb{N}$ (for the latter), which implies that neither of the 14 non-canonical maximum cliques has a design structure. $\square$ 

\section{Further discussion}\label{sec:discussion}
In this section, we discuss consequences of Theorem \ref{thm:main} and further research directions.

In \cite[p. 72, Table 3.3]{CD07}, the necessary and sufficient conditions for the existence of a 2-$(n,m,1)$ design with $m \le 9$ are given. For $m \ge 10$, much less is known. However, as was discussed in Section \ref{sec:Intro}, non-canonical maximum cliques in the block graph of a 2-$(n,m,1)$ design may exist only if $n \le m^3-2m^2+2m$. This means that for a fixed value of $m$, there exist only finitely many 2-$(n,m,1)$ designs whose block graphs may have non-canonical maximum cliques. We thus formulate the following open problem.
\begin{problem}\label{prob:InfFamily}
Does there exist an infinite family of 2-designs whose block graphs have non-canonical maximum cliques without a design structure?    
\end{problem}
Note that in such a family from Problem \ref{prob:InfFamily}, the parameter $m$ cannot be a fixed number. It follows from \cite[p. 103, Section 5.11]{CD07} that only four infinite families of 2-$(n,m,1)$ designs with a growing $m$ are known: 
\begin{itemize}
    \item the point-line incidence structure of an affine geometry, 
    \item the point-line incidence structure of a projective geometry, 
    \item unitals,
    \item Denniston designs.
\end{itemize}

Let $\operatorname{AG}(d,q)$ be an affine space and $D_1$ be the corresponding 2-$(n,m,1)$ design, where $n = q^d$ and $m = q$. If $d \ge 3$, the inequality $n > m^3 - 2m^2 + 2m$ is satisfied and the block graph of $D_1$ has only canonical maximum cliques. If $d = 2$, the inequality $n > m^3 - 2m^2 + 2m$ is not satisfied, but the design $D_1$ is given by the lines in an affine plane (the block graph is a complete multipartite graph, a trivial strongly regular graph, and all maximum cliques are easy to describe).

Let $\operatorname{PG}(d,q)$ be a projective space and $D_2$ be the corresponding 2-$(n,m,1)$ design, where $n = \frac{q^{d+1}-1}{q-1}$ and $m = q+1$. If $d \ge 4$, the inequality $n > m^3 - 2m^2 + 2m$ is satisfied and the block graph of $D_2$ has only canonical maximum cliques. If $d = 3$, we have $n = m^3 - 2m^2 + 2m$, and the block graph of $D_2$ is known as the Grassmann graph $J_q(4,2)$. The subgraph induced by the first neighbourhood of a vertex is the $q$-clique-extension of a $(q+1)\times(q+1)$-grid, which means that each vertex is contained in $q+1$ canonical maximum cliques and $q+1$ non-canonical maximum cliques (see \cite[Section 3.5.1]{BV22}). By \cite[Exercise 5.7]{GM15}, the non-canonical cliques in the block graph of $D_2$ correspond to the sets of lines in planes of $\operatorname{PG}(3,q)$. If $d=2$ then $D_2$ is a symmetric design and the corresponding block graph is a clique. 

Let $D_3$ be a unital, that is, a $2$-$(t^3+1,t+1,1)$ design, where $t$ is a positive integer. Then we have $n < m^3 - 2m^2 + 2m$ and the block graph of $D_3$ may have non-canonical maximum cliques. One of the most important infinite families of unitals is the family of Hermitian unitals. It is known \cite[Section 3.1.6]{BV22} that the block graph of a Hermitian unital has only canonical maximum cliques. Moreover, it was proved in \cite[Theorem 6.4]{D15} that the block graph of any unital has only canonical maximum cliques.

Let $D_4$ be a Denniston 2-$(2^{r+s}+2^r-2^s,2^r,1)$ design with $2 \le r < s$ (see \cite[p. 103]{CD07}, \cite[Section 2]{GR03} and \cite{D69}). It can be easily shown by direct substitution that the inequality $n \le m^3-2m^2+2m$ holds for $D_4$ (that is, the block graph of $D_4$ may have non-canonical cliques) if and only if $s < 2r$. We have computationally checked that for $(r,s) \in \{(2,3),(3,4),(3,5),(4,5),(4,6),(4,7)\}$ the block graph of a Denniston design has only canonical maximum cliques.
We thus formulate the following open problem.

\begin{problem}\label{prob:Denniston}
Does there exist a Denniston design whose block graph has a non-canonical maximum clique?  
\end{problem}

\section*{Acknowledgements} \label{Ack}
The authors thank R. Julian R. Abel for clarification the construction of the 2-(66,6,1) design from \cite[p. 76]{CD07}.
The authors also thank Anton Betten for providing the database of designs constructed in \cite{BBT03}. The authors are grateful to Sam Adriaensen for attracting their attention to the publication \cite{D15}. Sergey Goryainov thanks Three Gorges Mathematical Research Center for organising his visit in October 2023 during which these results were obtained. Finally, the authors thank the anonymous referees whose valuable comments significantly improved the paper.


\appendix
\section{Two more designs with the same parameters}
In this appendix we provide two more 2-(66,6,1) designs whose block graphs have non-canonical cliques without a subdesign structure. Since we do not know any symmetric description of these two designs, we assume that their sets of points are $\{1,\ldots, 66\}$ and just list the blocks.

The 143 blocks of the first additional design are:
\begin{longtable}{cccc}
   \{1,2,15,28,41,54\}, & \{1,3,16,29,42,55\}, & \{1,4,17,30,43,56\}, & \{1,5,18,31,44,57\}, \\
    \{1,6,19,32,45,58\}, &  \{1,7,20,33,46,59\}, & \{1,8,21,34,47,60\}, & \{1,9,22,35,48,61\},\\
    \{1,10,23,36,49,62\}, & \{1,11,24,37,50,63\}, & \{1,12,25,38,51,64\}, & \{1,13,26,39,52,65\},\\
    \{1,14,27,40,53,66\}, & \{2,11,14,18,23,25\}, & \{2,3,12,19,24,26\}, & \{3,4,13,20,25,27\}, \\
    \{4,5,14,15,21,26\}, & \{2,5,6,16,22,27\}, & \{3,6,7,15,17,23\}, & \{4,7,8,16,18,24\}, \\
    \{5,8,9,17,19,25\}, & \{6,9,10,18,20,26\}, & \{7,10,11,19,21,27\}, & \{8,11,12,15,20,22\}, \\
    \{9,12,13,16,21,23\}, & \{10,13,14,17,22,24\}, & \{2,13,29,30,35,46\}, & \{3,14,30,31,36,47\}, \\
    \{2,4,31,32,37,48\}, & \{3,5,32,33,38,49\}, & \{4,6,33,34,39,50\}, & \{5,7,34,35,40,51\}, \\
    \{6,8,28,35,36,52\}, & \{7,9,29,36,37,53\}, & \{8,10,30,37,38,41\}, & \{9,11,31,38,39,42\}, \\
    \{10,12,32,39,40,43\}, & \{11,13,28,33,40,44\}, & \{12,14,28,29,34,45\}, & \{2,10,34,59,63,65\}, \\
    \{3,11,35,60,64,66\}, & \{4,12,36,54,61,65\}, & \{5,13,37,55,62,66\}, & \{6,14,38,54,56,63\}, \\
    \{2,7,39,55,57,64\}, & \{3,8,40,56,58,65\}, & \{4,9,28,57,59,66\}, &  \{5,10,29,54,58,60\},\\
    \{6,11,30,55,59,61\}, & \{7,12,31,56,60,62\}, & \{8,13,32,57,61,63\}, & \{9,14,33,58,62,64\}, \\
    \{2,9,44,47,49,56\}, & \{3,10,45,48,50,57\}, & \{4,11,46,49,51,58\}, & \{5,12,47,50,52,59\}, \\
    \{6,13,48,51,53,60\}, & \{7,14,41,49,52,61\}, & \{2,8,42,50,53,62\}, & \{3,9,41,43,51,63\}, \\
    \{4,10,42,44,52,64\}, & \{5,11,43,45,53,65\}, & \{6,12,41,44,46,66\}, & \{7,13,42,45,47,54\}, \\
    \{8,14,43,46,48,55\}, & \{2,17,33,36,45,51\}, & \{3,18,34,37,46,52\}, & \{4,19,35,38,47,53\}, \\
    \{5,20,36,39,41,48\}, & \{6,21,37,40,42,49\}, & \{7,22,28,38,43,50\}, & \{8,23,29,39,44,51\}, \\
    \{9,24,30,40,45,52\}, & \{10,25,28,31,46,53\}, & \{11,26,29,32,41,47\}, & \{12,27,30,33,42,48\}, \\
    \{13,15,31,34,43,49\}, & \{14,16,32,35,44,50\}, & \{2,20,38,40,60,61\}, & \{3,21,28,39,61,62\}, \\
    \{4,22,29,40,62,63\}, & \{5,23,28,30,63,64\}, & \{6,24,29,31,64,65\}, & \{7,25,30,32,65,66\}, \\
    \{8,26,31,33,54,66\}, & \{9,27,32,34,54,55\}, & \{10,15,33,35,55,56\}, & \{11,16,34,36,56,57\}, \\
    \{12,17,35,37,57,58\}, & \{13,18,36,38,58,59\}, & \{14,19,37,39,59,60\}, & \{2,21,43,52,58,66\}, \\
    \{3,22,44,53,54,59\}, & \{4,23,41,45,55,60\}, & \{5,24,42,46,56,61\}, & \{6,25,43,47,57,62\}, \\
    \{7,26,44,48,58,63\}, & \{8,27,45,49,59,64\}, & \{9,15,46,50,60,65\}, & \{10,16,47,51,61,66\}, \\
    \{11,17,48,52,54,62\}, & \{12,18,49,53,55,63\}, & \{13,19,41,50,56,64\}, & \{14,20,42,51,57,65\}, \\
    \{15,27,29,38,52,57\}, & \{15,16,30,39,53,58\}, & \{16,17,31,40,41,59\}, & \{17,18,28,32,42,60\}, \\
    \{18,19,29,33,43,61\}, & \{19,20,30,34,44,62\}, & \{20,21,31,35,45,63\}, & \{21,22,32,36,46,64\}, \\
    \{22,23,33,37,47,65\}, & \{23,24,34,38,48,66\}, & \{24,25,35,39,49,54\}, & \{25,26,36,40,50,55\}, \\
    \{26,27,28,37,51,56\}, & \{15,25,37,44,45,61\}, & \{16,26,38,45,46,62\}, & \{17,27,39,46,47,63\}, \\
    \{15,18,40,47,48,64\}, & \{16,19,28,48,49,65\}, & \{17,20,29,49,50,66\}, & \{18,21,30,50,51,54\}, \\
    \{19,22,31,51,52,55\}, & \{20,23,32,52,53,56\}, & \{21,24,33,41,53,57\}, & \{22,25,34,41,42,58\}, \\
    \{23,26,35,42,43,59\}, & \{24,27,36,43,44,60\}, & \{15,24,32,51,59,62\}, & \{16,25,33,52,60,63\}, \\
    \{17,26,34,53,61,64\}, & \{18,27,35,41,62,65\}, & \{15,19,36,42,63,66\}, & \{16,20,37,43,54,64\}, \\
    \{17,21,38,44,55,65\}, & \{18,22,39,45,56,66\}, & \{19,23,40,46,54,57\}, & \{20,24,28,47,55,58\}, \\
    \{21,25,29,48,56,59\}, & \{22,26,30,49,57,60\}, & \{23,27,31,50,58,61\}. &  
\end{longtable}
The block graph of this design has 66 canonical and 14 non-canonical maximum cliques. 
The automorphism group of this design has order 39 and coincides with the automorphism group of the corresponding block graph.
The points split into four orbits of length 39, 13, 13 and 1, and the blocks split into five orbits of length 39, 39, 39, 13 and 13.
The canonical cliques split into four orbits of length 39,13,13 and 1, and the non-canonical cliques split into two orbits of length 13 and 1. A representative non-canonical clique of the orbit of length 13 is the following 13 blocks:
$$\{ 1, 2, 15, 28, 41, 54 \},$$
$$\{ 9, 27, 32, 34, 54, 55 \},$$
$$\{ 1, 3, 16, 29, 42, 55 \},$$
$$\{ 20, 24, 28, 47, 55, 58 \},$$
$$\{ 1, 6, 19, 32, 45, 58 \},$$
$$\{ 12, 14, 28, 29, 34, 45 \},$$
$$\{ 17, 18, 28, 32, 42, 60 \},$$
$$\{ 1, 8, 21, 34, 47, 60 \},$$
$$\{ 11, 26, 29, 32, 41, 47 \},$$
$$\{ 4, 23, 41, 45, 55, 60 \},$$
$$\{ 22, 25, 34, 41, 42, 58 \},$$
$$\{ 7, 13, 42, 45, 47, 54 \},$$
$$\{ 5, 10, 29, 54, 58, 60 \}.$$
Note that the set of intersecting points for these blocks is
$$
\{ 1, 28, 29, 32, 34, 41, 42, 45, 47, 54, 55, 58, 60 \}
$$
and has size 13. As in the proof of Proposition \ref{prop:C1}, the restriction of the 13 blocks above to the 13 intersecting points is again a $2$-(13,4,1) design, that is, a projective plane of order 3. So, we again have a 2-design,
and for each block a unique 2-point extension.

\bigskip
The 143 blocks of the second additional design are:
\begin{longtable}{cccc}
\{1,2,15,28,41,54\}, & \{1,3,16,29,42,55\}, & \{1,4,17,30,43,56\}, & \{1,5,18,31,44,57\}, \\
\{1,6,19,32,45,58\}, & \{1,7,20,33,46,59\}, & \{1,8,21,34,47,60\}, & \{1,9,22,35,48,61\}, \\
\{1,10,23,36,49,62\}, & \{1,11,24,37,50,63\}, & \{1,12,25,38,51,64\}, & \{1,13,26,39,52,65\}, \\
\{1,14,27,40,53,66\}, & \{2,11,14,18,23,25\}, & \{2,3,12,19,24,26\}, & \{3,4,13,20,25,27\}, \\
\{4,5,14,15,21,26\}, & \{2,5,6,16,22,27\}, & \{3,6,7,15,17,23\}, & \{4,7,8,16,18,24\}, \\
\{5,8,9,17,19,25\}, & \{6,9,10,18,20,26\}, & \{7,10,11,19,21,27\}, & \{8,11,12,15,20,22\}, \\
\{9,12,13,16,21,23\}, & \{10,13,14,17,22,24\}, & \{2,13,32,33,40,42\}, & \{3,14,28,33,34,43\}, \\
\{2,4,29,34,35,44\}, & \{3,5,30,35,36,45\}, & \{4,6,31,36,37,46\}, & \{5,7,32,37,38,47\}, \\
\{6,8,33,38,39,48\}, & \{7,9,34,39,40,49\}, & \{8,10,28,35,40,50\}, & \{9,11,28,29,36,51\}, \\
\{10,12,29,30,37,52\}, & \{11,13,30,31,38,53\}, & \{12,14,31,32,39,41\}, & \{2,9,31,56,64,66\}, \\
\{3,10,32,54,57,65\}, & \{4,11,33,55,58,66\}, & \{5,12,34,54,56,59\}, & \{6,13,35,55,57,60\}, \\
\{7,14,36,56,58,61\}, & \{2,8,37,57,59,62\}, & \{3,9,38,58,60,63\}, & \{4,10,39,59,61,64\}, \\
\{5,11,40,60,62,65\}, & \{6,12,28,61,63,66\}, & \{7,13,29,54,62,64\}, & \{8,14,30,55,63,65\}, \\
\{2,10,45,47,51,63\}, & \{3,11,46,48,52,64\}, & \{4,12,47,49,53,65\}, & \{5,13,41,48,50,66\}, \\
\{6,14,42,49,51,54\}, & \{2,7,43,50,52,55\}, & \{3,8,44,51,53,56\}, & \{4,9,41,45,52,57\}, \\
\{5,10,42,46,53,58\}, & \{6,11,41,43,47,59\}, & \{7,12,42,44,48,60\}, & \{8,13,43,45,49,61\}, \\
\{9,14,44,46,50,62\}, & \{2,17,36,39,53,60\}, & \{3,18,37,40,41,61\}, & \{4,19,28,38,42,62\}, \\
\{5,20,29,39,43,63\}, & \{6,21,30,40,44,64\}, & \{7,22,28,31,45,65\}, & \{8,23,29,32,46,66\}, \\
\{9,24,30,33,47,54\}, & \{10,25,31,34,48,55\}, & \{11,26,32,35,49,56\}, & \{12,27,33,36,50,57\}, \\
\{13,15,34,37,51,58\}, & \{14,16,35,38,52,59\}, & \{2,20,30,48,49,58\}, & \{3,21,31,49,50,59\}, \\
\{4,22,32,50,51,60\}, & \{5,23,33,51,52,61\}, & \{6,24,34,52,53,62\}, & \{7,25,35,41,53,63\}, \\
\{8,26,36,41,42,64\}, & \{9,27,37,42,43,65\}, & \{10,15,38,43,44,66\}, & \{11,16,39,44,45,54\}, \\
\{12,17,40,45,46,55\}, & \{13,18,28,46,47,56\}, & \{14,19,29,47,48,57\}, & \{2,21,38,46,61,65\}, \\
\{3,22,39,47,62,66\}, & \{4,23,40,48,54,63\}, & \{5,24,28,49,55,64\}, & \{6,25,29,50,56,65\}, \\
\{7,26,30,51,57,66\}, & \{8,27,31,52,54,58\}, & \{9,15,32,53,55,59\}, & \{10,16,33,41,56,60\}, \\
\{11,17,34,42,57,61\}, & \{12,18,35,43,58,62\}, & \{13,19,36,44,59,63\}, & \{14,20,37,45,60,64\}, \\
\{15,25,36,40,47,52\}, & \{16,26,28,37,48,53\}, & \{17,27,29,38,41,49\}, & \{15,18,30,39,42,50\}, \\
\{16,19,31,40,43,51\}, & \{17,20,28,32,44,52\}, & \{18,21,29,33,45,53\}, & \{19,22,30,34,41,46\}, \\
\{20,23,31,35,42,47\}, & \{21,24,32,36,43,48\}, & \{22,25,33,37,44,49\}, & \{23,26,34,38,45,50\}, \\
\{24,27,35,39,46,51\}, & \{15,24,29,31,60,61\}, & \{16,25,30,32,61,62\}, & \{17,26,31,33,62,63\}, \\
\{18,27,32,34,63,64\}, & \{15,19,33,35,64,65\}, & \{16,20,34,36,65,66\}, & \{17,21,35,37,54,66\}, \\
\{18,22,36,38,54,55\}, & \{19,23,37,39,55,56\}, & \{20,24,38,40,56,57\}, & \{21,25,28,39,57,58\}, \\
\{22,26,29,40,58,59\}, & \{23,27,28,30,59,60\}, & \{15,27,45,48,56,62\}, & \{15,16,46,49,57,63\}, \\
\{16,17,47,50,58,64\}, & \{17,18,48,51,59,65\}, & \{18,19,49,52,60,66\}, & \{19,20,50,53,54,61\}, \\
\{20,21,41,51,55,62\}, & \{21,22,42,52,56,63\}, & \{22,23,43,53,57,64\}, & \{23,24,41,44,58,65\}, \\
\{24,25,42,45,59,66\}, & \{25,26,43,46,54,60\}, & \{26,27,44,47,55,61\}. &  \\
\end{longtable}
The block graph of this design has 66 canonical and 14 non-canonical maximum cliques. 
The automorphism group of this design has order 39 and coincides with the automorphism group of the corresponding block graph.
The points split into four orbits of length 39, 13, 13 and 1, and the blocks split into five orbits of length 39, 39, 39, 13 and 13.
The canonical cliques split into four orbits of length 39,13,13 and 1, and the non-canonical cliques split into two orbits of length 13 and 1. A representative non-canonical clique of the orbit of length 13 is the following 13 blocks:
$$\{ 4, 19, 28, 38, 42, 62 \},$$
$$\{ 8, 26, 36, 41, 42, 64 \},$$
$$\{ 7, 13, 29, 54, 62, 64 \},$$
$$\{ 1, 2, 15, 28, 41, 54 \},$$
$$\{ 20, 21, 41, 51, 55, 62 \},$$
$$\{ 1, 3, 16, 29, 42, 55 \},$$
$$\{ 9, 11, 28, 29, 36, 51 \},$$
$$\{ 17, 27, 29, 38, 41, 49 \},$$
$$\{ 6, 14, 42, 49, 51, 54 \},$$
$$\{ 18, 22, 36, 38, 54, 55 \},$$
$$\{ 1, 10, 23, 36, 49, 62 \},$$
$$\{ 5, 24, 28, 49, 55, 64 \},$$
$$\{ 1, 12, 25, 38, 51, 64 \}.$$
Note that the set of intersecting points for these blocks is
$$
\{ 1, 28, 29, 36, 38, 41, 42, 49, 51, 54, 55, 62, 64 \}
$$
and has size 13. As in the proof of Proposition \ref{prop:C1}, the restriction of the 13 blocks above to the 13 intersecting points is again a $2$-(13,4,1) design, that is, a projective plane of order 3. So, we again have a 2-design,
and for each block a unique 2-point extension.

Thus, the two designs above have the same lengths of orbits as the design described in the main part of this paper. We thus ask a question: is there a symmetric description of the two designs above, similar to the description of the design considered in the main part of this paper?    
\end{document}